\documentclass[12pt]{amsart}

\usepackage{amsrefs}
\usepackage{amssymb}
\usepackage{fancyhdr}
\usepackage{bbm}
\usepackage{xcolor}
\usepackage{hyperref}
\hypersetup{
  colorlinks=true,
  linkcolor=blue,
  citecolor=blue
}
\usepackage{mleftright}

\newtheorem{theorem}{Theorem}
\newtheorem*{theorem*}{Theorem}
\newtheorem{lemma}[theorem]{Lemma}
\newtheorem{proposition}[theorem]{Proposition}
\newtheorem{claim}[theorem]{Claim}

\newtheorem{maintheorem}{Theorem}

\theoremstyle{definition}
\newtheorem{remark}[theorem]{Remark}

\newtheorem*{definition*}{Definition}
\newtheorem*{lemma*}{Lemma}
\usepackage{graphicx}
\usepackage{pstricks, enumerate, pst-node, pst-text, pst-plot}

\numberwithin{equation}{section}
\numberwithin{theorem}{section}

\newcommand{\interior}[1]{%
  {\kern0pt#1}^{\mathrm{o}}%
}

\newcommand{\R}{\mathbb{R}}
\newcommand{\N}{\mathbb{N}}

\usepackage{xparse}
\DeclareDocumentCommand\Pr{ m g }{\ensuremath{
    {   \IfNoValueTF {#2}
      {\mathbb{P}\mleft[{#1}\mright]}
      {\mathbb{P}\mleft[{#1}\middle\vert{#2}\mright]}%
    }
}}
\DeclareDocumentCommand\E{ m g }{\ensuremath{
    {   \IfNoValueTF {#2}
      {\mathbb{E}\mleft[{#1}\mright]}
      {\mathbb{E}\mleft[{#1}\middle\vert{#2}\mright]}%
    }
}}

\DeclareMathOperator{\supp}{supp}

\def\dd{\mathrm{d}}

\def\cM{\mathcal{M}}
\def\bF{\mathbb{F}}
\begin{document}

\title[]{Infinite stationary measures of co-compact group actions}

\author[]{Mohammedsaid Alhalimi}
\address{UC Berkeley, Department of Mathematics}

\author[]{Tom Hutchcroft}
\address{California Institute of Technology, Division of Physics, Mathematics and Astronomy}

\author[]{Minghao Pan}
\address{California Institute of Technology, Division of Physics, Mathematics and Astronomy}

\author[]{Omer Tamuz}
\address{California Institute of Technology, Division of Physics, Mathematics and Astronomy}

\author[]{Tianyi Zheng}
\address{UC San Diego, Department of Mathematics}

%    Information for third author

%\thanks{}
\date{\today}

\begin{abstract}
  Let $\Gamma$ be a finitely generated group, and let $\mu$ be a nondegenerate, finitely supported probability measure on $\Gamma$. We show that every co-compact $\Gamma$ action on  a locally compact Hausdorff space admits a nonzero $\mu$-stationary Radon measure. The main ingredient of the proof is a stationary analogue of Tarski's theorem: we show that for every nonempty subset $A \subseteq \Gamma$ there is a $\mu$-stationary, finitely additive measure on $\Gamma$ that assigns unit mass to $A$.
\end{abstract}

\maketitle
%\tableofcontents
\section{Introduction}

Let $\Gamma$ be a finitely generated group, and let $\mu$ be a finitely supported probability measure on $\Gamma$ that is nondegenerate, i.e., whose support generates $\Gamma$ as a semigroup. Let $X$ be a Hausdorff topological space, and suppose that $\Gamma$ acts on $X$ by homeomorphisms.

When $X$ is compact, there is always a $\mu$-stationary probability measure on $X$, that is, a probability measure $\lambda$ such that $\sum_g \mu(g) g\lambda = \lambda$; this follows from the Markov-Kakutani fixed point theorem. Invariant probability measures are a particular case of stationary measures, but invariant ones may not exist when $\Gamma$ is nonamenable. Indeed, $\Gamma$ is amenable if and only if every action of $\Gamma$ on a compact Hausdorff space admits an invariant probability measure.

In this paper we go beyond the compact case, and consider locally compact spaces. Here, there may not be a stationary probability measure, so instead we focus on infinite measures. But in this class one can always find a stationary measure---indeed, an invariant one---by taking the counting measure on an orbit. Accordingly, we restrict our attention to Radon measures (in particular, measures that assign finite mass to compact sets). The construction of stationary, infinite Radon measures has attracted attention in particular cases; see., e.g., the case of the action of subgroups of $\mathrm{Homeo}^+(\R)$ on $\R$ in \cite{Bougerol1995,deroin2013symmetric}.

Kellerhals,  Monod,  R{\o }rdam~\cite{kellerhals2013non} ask a similar question, but consider invariant measures rather than stationary ones. Focusing on co-compact actions---i.e., actions for which there exists a compact $K \subseteq X$ such that $X = \cup_g g K$. They show that every co-compact action of $\Gamma$ on a locally compact Hausdorff space admits an invariant nonzero Radon measure if and only if $\Gamma$ is \emph{supramenable}.\footnote{Following Rosenblatt \cite{rosenblatt1974invariant}, a group is said to be \emph{supramenable} if it does not have any paradoxical subsets, or equivalently if  for every action of $\Gamma$ on a set $X$ and every subset $A\subseteq X$ there exists a finitely additive $\Gamma$-invariant measure $M$ with $M(A)=1$; examples of groups that are amenable but not supramenable include the lamplighter group and Baumslag-Solitar groups.
}
Our first main result shows that, as in the case of compact $X$, stationary measures always exist.

\begin{maintheorem}
  \label{thm:radon}
  Suppose $\Gamma$ acts co-compactly on a locally compact, Hausdorff space $X$. Then there is a nonzero, $\mu$-stationary Radon measure on $X$.
\end{maintheorem}
In particular, it follows that if $\Gamma$ is not supramenable, then it admits a co-compact action on a locally compact Hausdorff space with a $\mu$-stationary measure, but no $\Gamma$-invariant measure. 

The main component of the proof of Theorem~\ref{thm:radon} is our next main result, Theorem~\ref{thm:subsets}; using Theorem~\ref{thm:subsets}, Theorem~\ref{thm:radon} follows by a functional-analytic argument adapted from~\cite{kellerhals2013non}. Theorem~\ref{thm:subsets} tackles a stationary analogue of a question posed by von Neumann~\cite[\S 4]{vneumann1929zur}. He asked, given a subset $A \subseteq \Gamma$, when does there exists a $\Gamma$-invariant, finitely additive (perhaps infinite) measure on $\Gamma$ that assigns unit mass to $A$? Famously, the answer given by Tarski is that such a measure exists if and only if $A$ is not paradoxical~\cite{tarski1949car}. We ask---and answer---the same question, but for stationary rather than invariant measures.  
\begin{maintheorem}
    \label{thm:subsets}
    For every nonempty subset $A \subseteq \Gamma$ there is a $\mu$-stationary, finitely additive measure on $\Gamma$ that assigns unit mass to $A$.
\end{maintheorem} 
The proof of Theorem~\ref{thm:subsets} considers two cases. The first is the case that the $\mu$-random walk on $\Gamma$ visits $A$ infinitely often, in expectation. In this case one can essentially apply the same argument that is used in the proof of the Markov-Kakutani fixed point theorem. The more interesting case is the complement, where the measure constructed using the Markov-Kakutani argument may fail to be stationary (Remark~\ref{remark:mk}). In this case, we instead construct our stationary measure using the Green function of the $\mu$-random walk. The interesting challenge (which does not arise in the similar construction of the Martin boundary) is to control the deviations from stationarity (see Remark~\ref{remark:errors}).

\subsection{$\mu$-light subsets of $\Gamma$}
\label{sec:small}
The measures we construct in Theorem~\ref{thm:subsets} are in general not countably additive, but may have this property in some cases. We say that a subset $A \subseteq \Gamma$ is $\mu$-light if there exists a countably additive, $\mu$-stationary measure $M$ on $\Gamma$ such that $M(A)=1$. Note that if $M$ is a countably additive measure on $\Gamma$ then there is a (unique) function $f \colon \Gamma \to \R$ such that $M(E) = \sum_{g \in E}f(g)$. If $M$ is nonzero and $\mu$-stationary, then $f$ is positive and $\mu * f = f$, i.e., $f$ is a positive (left) $\mu$-harmonic function. Thus $A \subseteq \Gamma$ is $\mu$-light if there exists a positive $\mu$-harmonic $f$ such that $\sum_{g \in A}f(g) < \infty$. It follows that  if $A$ is $\mu$-light then every subset of $A$ is $\mu$-light; this provides some justification for calling these sets $\mu$-light.

Another justification is that every finite $A$ is $\mu$-light, since one can take $M$ to be the counting measure on $\Gamma$, divided by the size of $A$. On the other hand, if $\Gamma$ is infinite then $\Gamma$ itself is not $\mu$-light. To see this, suppose $M(\{g\}) = f(g)$. Since $f$ is $\mu$-harmonic, by the maximum principle $f$ is either constant or else does not attain its maximum. In both cases there are infinitely many $g$ such that $f(g) \geq f(e)$, and so $M(\Gamma)$ is infinite.

\bigskip

Since $\mu$-light sets play a special role in Theorem~\ref{thm:subsets}, it is interesting to describe their structure geometrically.
We study $\mu$-light sets in the particular case of $\Gamma = \bF_d$, the free group on $d \geq 2$ generators, and where $\mu$ is the uniform distribution on the generators: the measure that assigns mass $1/(2d)$ to each of the generators and their inverses. 

Let $B_r \subset \bF_d$ be the ball of radius $r$ in $\Gamma$, under the standard word length metric defined by the generators. Given a subset $A \subseteq \bF_d$ we denote the lower and upper exponential growth rates of $A$ by
\begin{align*}
    \underline{\mathfrak{g}}(A) = \liminf_r |A \cap B_r|^{1/r}\quad\quad\text{ and }\quad\quad~~\overline{\mathfrak{g}}(A) = \limsup_r |A \cap B_r|^{1/r}.
\end{align*}

\begin{proposition}
  \label{prop:small}
  Suppose $\overline{\mathfrak{g}}(A) < \sqrt{2d-1}$. Then $A$ is $\mu$-light.
\end{proposition}
We prove this proposition by showing that if we draw a $\mu$-harmonic function $f$ at random from the hitting measure on the Martin boundary of $\bF_d$ then $f(A) < \infty$ almost surely. To this end, we use a technique explored in~\cite{hutchcroft2019statistical}, involving the study of the expectation of the square root of the randomly chosen harmonic function.

As a partial converse to this proposition we show that this bound is tight, i.e., that there exists an $A\subseteq \bF_d$ with $\underline{\mathfrak{g}}=\overline{\mathfrak{g}}=\sqrt{2d-1}$ that is not $\mu$-light. Denote by $|g|$ the standard word length norm. Let $\sigma \colon \bF_d \to \bF_d$ be an injection such that $|\sigma(g)| = 2|g|$, and, in the Cayley graph defined by the generators, $g$ lies on the unique shortest path from $\sigma(g)$ to $e$ (i.e.\ $g$ is a prefix of $\sigma(g)$). For example, given $g = s_1 s_2 \cdots s_r$, where each $s_i$ is a generator or its inverse, one can take 
\begin{align}
  \label{eq:sigma}
    \sigma(s_1 s_2 \cdots s_r) = s_1 s_2 \cdots s_r^{r+1}.
\end{align}
Another example is
\begin{align*}
    \sigma(s_1 s_2 \cdots s_r) = s_1 s_2 \cdots s_r s_r s_{r-1} \cdots s_1.
\end{align*}
so that $\sigma(\bF_d)$ is the set of palindromes.
\begin{claim}
  \label{clm:sigma}
    Let $A = \sigma(\bF_d)$. Then $\underline{\mathfrak{g}}=\overline{\mathfrak{g}}=\sqrt{2d-1}$, and $A$ is not $\mu$-light.
\end{claim}

\subsection{Necessity of the co-compactness assumption}

Theorem~\ref{thm:radon} shows that stationary measures exist for co-compact actions. A natural question is whether this result applies more generally.

The same question was addressed by Matui and R{\o}rdam \cite{matui2015universal}, regarding invariant measures. They show that in the invariant setting the co-compactness assumption is necessary: every infinite group admits an action on a locally compact space that has no invariant nonzero Radon measures. This implies the following result on $\mu$-stationary measures:
\begin{proposition}
    \label{prop:vn}
    Let $\Gamma$ be infinite and virtually nilpotent, and let $\mu$ be nondegenerate and symmetric. Then there exists an action of $\Gamma$ on a locally compact, $\sigma$-compact Hausdorff space $X$ that admits no nonzero $\mu$-stationary Radon measures.
\end{proposition}
\begin{proof}
  By \cite[Proposition 4.3]{matui2015universal}, there is a locally compact, $\sigma$-compact Hausdorff $\Gamma$-space $X$ that admits no nonzero $\Gamma$-invariant Radon measures. Suppose $\lambda$ is a $\mu$-stationary Radon measure on $X$. Fix any $Y \subseteq X$  such that $\lambda(Y) < \infty$. Then the function $f \colon \Gamma \to \R$ given by $f(g) = [g\lambda](Y)$ is a non-negative $\mu$-harmonic function on $G$. By a result of Margulis~\cite{margulis1966positive}, $f$ must be constant, because $\Gamma$ is virtually nilpotent and $\mu$ is symmetric (indeed, Margulis shows that if $\Gamma$ is virtually nilpotent and $\mu$ is symmetric then all non-negative $\mu$-harmonic functions are constant). Hence $[g\lambda](Y) = \lambda(Y)$, and, since $Y$ is general, $\lambda$ is $\Gamma$-invariant. It thus follows that $\lambda$ is the zero measure.
\end{proof}

Proposition~\ref{prop:vn} shows that the co-compactness assumption is necessary for Theorem~\ref{thm:radon}, in the case of virtually nilpotent $\Gamma$ and symmetric $\mu$. The next proposition shows that it is also necessary in the case of $\Gamma = \bF_d$, the free group on $d \geq 2$ generators, and where $\mu$ is the uniform measure on the generators (i.e., as above, the measure that assigns mass $1/(2d)$ to each of the generators and their inverses).
\begin{proposition}
    \label{prop:free-no-stationary}
    Let $\Gamma=\bF_d$ for $d \geq 2$, and let $\mu$ be the uniform measure on the generators. Then there exists an action of $\Gamma$ on a locally compact, second countable, $\sigma$-compact Hausdorff space $X$ that admits no nonzero $\mu$-stationary Radon measures.
\end{proposition}
The proof adapts the construction of Matui and R{\o}rdam \cite{matui2015universal} from the invariant case to the stationary one. We note that by following their construction more closely one can furthermore show that the $G$ action on $X$ can be taken to be free.

\subsection{Open questions}

Our proof of Theorem~\ref{thm:subsets} crucially relies on the assumption that $\mu$ is finitely supported. However, it seems plausible that both Theorem~\ref{thm:radon} and Theorem~\ref{thm:subsets} hold more generally. Beyond countable groups, similar statement could apply to actions of (say) locally compact, second countable groups, with sufficiently regular measures $\mu$.

In Propositions~\ref{prop:vn} and~\ref{prop:free-no-stationary} we show that for infinite virtually nilpotent groups, as well as for non-abelian free groups, the co-compactness assumption of Theorem~\ref{thm:radon} cannot be dropped. We expect that this is true for every infinite group.

\section{Definitions}
Given a function $f \colon \Gamma \to \R$ and $k \in \Gamma$ we define the left and right translations of $f$ by
\begin{align*}
    [k f](g) &= f(k^{-1}g)\\
    f^k(g) &= f(g k).
\end{align*}
For $k \in \Gamma$ the function $\delta_k \colon \Gamma \to \{0,1\}$ is the characteristic function of $\{k\}$. Note that $\delta_k = k\delta_e = \delta_e^{k^{-1}}$.

We denote convolution of measures by $*$, and denote by $\mu^{(n)}$ the $n$-fold convolution of $\mu$ with itself. We denote by $\mu * f$ the convolution of $\mu$ with the function $f$:
\begin{align*}
    [\mu * f](g) = \sum_h \mu(h) [h f](g) = \sum_h \mu(h) f(h^{-1} g).
\end{align*}
We say that $f$ is (left) $\mu$-harmonic if $\mu * f = f$.

Given an action of $\Gamma$ on a set $S$, and given a finitely additive measure $M$ on $S$, we denote by $g M$ the measure $[g M](A) = M(g^{-1}A)$. Given a finitely supported measure $\mu$ on $\Gamma$, we denote $\mu * M = \sum_{g \in \Gamma}\mu(g) g M$. We say that $M$ is $\mu$-stationary if $\mu * M = M$. The same definition applies to the special case of a Radon measure $\lambda$ on a topological space $X$ that $\Gamma$ acts on.

A non-negative function $f \colon \Gamma \to [0,\infty]$ can be identified with the countably additive measure on $\Gamma$ that assigns to $E \subseteq \Gamma$ the measure $\sum_{g \in E}f(g)$. We overload notation and also denote this measure by $f$, so that $f(E) = \sum_{g \in E}f(g)$.

\section{Proofs}

Since Theorem~\ref{thm:subsets} is used in the proof of Theorem~\ref{thm:radon}, we prove the former first.

\subsection{Proof of Theorem~\ref{thm:subsets}}

Given $\mu$, define the Green function $G \colon \Gamma \to [0,\infty]$ by
\begin{align*}
  G(g)=  \sum_{n=0}^\infty \mu^{(n)}(g).
\end{align*}
Since $\mu$ generates $\Gamma$ as a semigroup we have that $G > 0$. Note that for every $h \in \Gamma$ there is a constant $\varepsilon_h > 0$ such that for all $k \in \Gamma$
\begin{align}
    \label{eq:gamma_bounds}
    \varepsilon_hG^k \leq hG^k \leq \frac{1}{\varepsilon_h} G^k~~\text{ and }~~\varepsilon_h kG \leq kG^h \leq \frac{1}{\varepsilon_h} kG.
\end{align}

\begin{lemma}
    \label{lem:gamma}
    Let $A$ be a subset of $\Gamma$ such that $G(A) < \infty$. Then 
    \begin{align*}
        \inf \{ G^k(A) \,:\, k \in \Gamma\} = 0.
    \end{align*}
\end{lemma}
\begin{proof}
    Suppose towards a contradiction that $\inf \{G^{k}(A)\,:\, k \in \Gamma\} = c > 0$. Then for any $m \in \N$
    \begin{align*}
        \sum_{n=m}^\infty \mu^{(n)}(A)
        &= \sum_{n=0}^\infty [\mu^{(n)} * \mu^{(m)}](A)\\
        &= \sum_{n=0}^\infty \sum_{k \in \Gamma}\mu^{(n)}(A k) \cdot \mu^{(m)}(k^{-1})\\
        &= \sum_{k \in \Gamma}\mu^{(m)}(k^{-1})\sum_{n=0}^\infty \mu^{(n)}(A k) \\
        &= \sum_{k \in \Gamma}\mu^{(m)}(k^{-1})G^k(A) \\
        &\geq c.
    \end{align*}
    Hence
    \begin{align*}
        G(A) = \sum_{n=0}^\infty \mu^{(n)}(A) = \infty,
    \end{align*}
    and we have arrived at a contradiction.
\end{proof}

Recall that a finitely additive measure on $\Gamma$ is a function $M \colon 2^\Gamma \to [0,\infty]$ such that for all disjoint $A,B \subseteq \Gamma$ it holds that $M(A\cup B) = M(A)+M(B)$. Endow the set $[0,\infty]^{2^\Gamma}$ of all functions $2^\Gamma \to [0,\infty]$ on subsets on $\Gamma$ with the product topology, which is compact. Denote by $\cM \subset [0,\infty]^{2^\Gamma}$ the set of finitely additive measures. This is a closed subset, and hence also compact (but not necessarily sequentially compact). Given $A \subseteq \Gamma$, the set of measures $M \subseteq \cM$ such that $M(A)=1$ is a closed subset of $\cM$. 

Denote by $\cM_\mu$ the $\mu$-stationary measures $\{M \in \cM \,:\, \mu * M = M\}$. If $\mu$ is finitely supported, then $\cM_\mu$ is closed. To see this, note that 
\begin{align*}
    \cM_\mu = \bigcap_{E \subseteq \Gamma} \cM_\mu^E,
\end{align*}
where
\begin{align*}
    \cM_\mu^E = \left\{ M \in \cM\,:\, \sum_{g \in \supp \mu}\mu(g)M(g^{-1}E) = M(E)\right\}.
\end{align*}
Since $\mu$ is finitely supported, membership in $\cM_\mu^E$ depends on only finitely many coordinates $(g^{-1}E)_{g \in \supp E}$. Furthermore, its projection to these coordinates is a closed subset of $[0,\infty]^{\supp \Gamma}$. Hence each set in the intersection is closed, and thus their intersection is closed.

By the same argument, for $\varepsilon>0$ the set of measures
\begin{align*}
    \cM_\mu^{E,\varepsilon} = \left\{ M \in \cM\,:\, \sum_{g \in \supp \mu}\mu(g)M(g^{-1}E) \in M(E) + [-\varepsilon,+\varepsilon] \right\}
\end{align*}
is closed. Clearly, $\cM_\mu = \cap_{E,\varepsilon}\cM_\mu^{E,\varepsilon}$.

Suppose that $M \in \cM$ is a limit point of $\{M_1, M_2, \ldots\}$. To show that $M \in \cM_\mu$ (i.e., that $M$ is $\mu$-stationary), it suffices to show that $M \in \cM_\mu^{E,\varepsilon}$ for all $E \subseteq \Gamma$ and $\varepsilon >0$.

\begin{proof}[Proof of Theorem~\ref{thm:subsets}]
    As above, the Green function is given by $G(g) = \sum_{n=0}^\infty \mu^{(n)}(g)$, and we denote $G(A) = \sum_{g \in A}G(g)$.

    We consider two cases: either $G(A) = \infty$ or $G(A) < \infty$. 
    
    Suppose first that $G(A) = \infty$. For $n \in \N$ define the countably additive measure
    \begin{align}
        \label{eq:MK}
        M_n(E) = \frac{\sum_{m=0}^n \mu^{(m)}(E)}{\sum_{m=0}^n \mu^{(m)}(A)},
    \end{align}
    and let $M$ be a limit point of $\{M_1,M_2,\ldots\}$ in the (compact) space of finitely additive measures on $\Gamma$. Since for every $E \subseteq \Gamma$
    \begin{align*}
        |[\mu * M_n](E) - M_n(E)|
        &= \frac{|\mu^{(0)}(E) - \mu^{(n+1)}(E)|}{\sum_{m=0}^n \mu^{(m)}(A)}\\
        &\leq \frac{1}{\sum_{m=0}^n \mu^{(m)}(A)},
    \end{align*}
    and since the denominator tends to infinity by the assumption that $G(A)=\infty$, we have that $M \in \cM_\mu^{E,\varepsilon}$ for all $\varepsilon > 0$. Since this holds for every $E$, $M \in \cM_\mu$.

    Consider now the case that $G(A) < \infty$. In this case the argument above fails, as taking the limit in \eqref{eq:MK} yields the measure $M(E) = G(E)/G(A)$, which cannot be stationary (see Remark~\ref{remark:mk}). Instead, we pursue a different approach.
    
    Given $k \in \Gamma$ define the countably additive measure $M_k$ on $\Gamma$ by
    \begin{align*}
        M_k(E) = \frac{G^{k}(E)}{G^{k}(A)}.
    \end{align*}
    We have by \eqref{eq:gamma_bounds} that $0 < G^k(A) < \infty$, and so $M_k$ is well-defined with $M_k(A)=1$. Note that $G - [\mu * G] = \delta_e$, and more generally $G^k - [\mu * G]^k = \delta_{k^{-1}}$. Hence
    \begin{align}
        \label{eq:M_k}
        M_k(E) - [\mu * M_k](E) =  \frac{G^k(E) - [\mu * G^k](E)}{G^{k}(A)} = \frac{\delta_{k^{-1}}(E)}{G^{k}(A)}.
    \end{align}
    To complete the proof, we need to find a sequence of choices of $k$ such that this error vanishes. Of course, if we could ensure that $G^{k}(A)$ tends to infinity we would achieve this goal; however, this is in general not possible. Instead, our proof ``paradoxically'' relies on a sequence with the opposite property, namely that $G^{k}(A)$ vanishes (see Remark~\ref{remark:errors}).

    By Lemma~\ref{lem:gamma} there is a sequence  $k_1,k_2,\ldots$ in $\Gamma$ such that $ G^{k_n}(A) \to 0$ as $n\to\infty$. Let $M$ be a limit point of $\{M_{k_1},M_{k_2},\ldots\}$. To complete the proof we show that $M(E) = [\mu * M](E)$ for all $E \subseteq \Gamma$. 

    Consider first the case that $M(E) = \infty$. By \eqref{eq:gamma_bounds}, it holds for any $k \in \Gamma$ and $E \subseteq \Gamma$ that
    \begin{align*}
        [\mu * M_k](E) 
        &=  \frac{[\mu * G^k](E)}{G^{k}(A)}\\
        &= \frac{\sum_h \mu(h)[h G^k](E)}{G^{k}(A)}\\
        &\geq \frac{\sum_h \mu(h)\varepsilon_h  G^k(E)}{G^{k}(A)}\\
        &= M_k(E)\sum_h \mu(h)\varepsilon_h.
    \end{align*}
    Hence $[\mu * M](E) \geq M(E)\sum_h \mu(h)  = \infty$, and so $M \in \cM_\mu^E$. 
    % Since this holds for every $E$, $M \in \cM_\mu$.

    Finally, consider the case that $M(E) < \infty$. Since $M$ is a limit point of $\{M_{k_1},M_{k_2},\ldots\}$, and since $\Gamma$ is countable, there is a  a subsequence $h_1,h_2,\ldots$ of $k_1,k_2,\ldots$ such that $\lim_n M_{h_n}(g E) = M(g E)$ for all $g \in \Gamma$. 
    (Since we are working with spaces that might not be sequentially compact, this subsequence may need to depend on the choice of $E$. This does not cause any problems.)
    In particular, $\lim_n M_{h_n}(E) = M(E)$ and $\lim_n [\mu * M_{h_n}](E) = [\mu * M](E)$.
    Since 
    \begin{align*}
      G^{k}(E) = G(E k) = \sum_g G(g)\delta_g(E k) \geq G(e) \delta_e(E k) = G(e)\delta_{k^{-1}}(E),
    \end{align*}
    it follows that
    \begin{align*}
        \limsup_n \frac{G(e)\delta_{h_n^{-1}}(E)}{G^{h_n}(A)} \leq \limsup_n \frac{G^{h_n}(E)}{G^{h_n}(A)}  = \lim_n M_{h_n}(E) = M(E) < \infty.
    \end{align*}
    Since $\lim_n G^{h_n}(A) = 0$ by construction and $G(e)>0$, we can conclude that $\delta_{h_n^{-1}}(E) = 0$ for all $n$ large enough. It then follows from \eqref{eq:M_k} that
    \begin{align}
        \label{eq:M_h}
        |[\mu * M_{h_n}](E) - M_{h_n}(E)| = \frac{\delta_{h_n^{-1}}(E)}{G^{h_n}(A)}.
    \end{align}
    But $\delta_{h_n^{-1}}(E) = 0$ for all $n$ large enough, and so
    \begin{align*}
        \lim_n|[\mu * M_{h_n}](E) - M_{h_n}(E)| = 0.
    \end{align*}
    Hence $[\mu * M_{h_n}](E) = M_{h_n}(E)$, and $M \in \cM_\mu^{E,\varepsilon}$ for all $\varepsilon > 0$. Since this holds for every $E$, $M \in \cM_\mu$.
\end{proof}
We end this section with two remarks regarding this proof.
\begin{remark}\label{remark:mk}
  Note that the measure constructed by taking the limits in \eqref{eq:MK} may fail to be stationary when the Green function assigns finite measure to $A$, and so a different argument is indeed needed in that case. For example, consider, as above, the case that $\mu$ is the step distribution of a transient random walk on some group $\Gamma$. Suppose $G(A)<\infty$. Then the limit $M$ of the measures $M_n$ given by \eqref{eq:MK} is given by $M(E) = G(E)/G(A)$, where $G$ is the Green function. Since $G$ vanishes at infinity, it is maximized at some $g$ and hence cannot be stationary, by the maximum principle. 
\end{remark}
\begin{remark}
\label{remark:errors}
In the case that the Green function $G$ assigns finite measure to the set $A$,  the error term that needs to be controlled is $\delta_{h_n^{-1}}(E)/G^{h_n}(A)$: we need the limit of this expression to vanish as $n$ tends to infinity; see \eqref{eq:M_h}. Of course, this would indeed be the case if $G^{h_n}(A)$ were to tend to infinity with $n$. However, we cannot guarantee that there exists a sequence $(h_n)_n$ with this property. Instead, we rely on the existence of a sequence with an opposite property, namely that $G^{h_n}(A)\to_n 0$. As this proof shows, this too works, because along such a sequence it is guaranteed that the denominator $\delta_{h_n^{-1}}(E)$ is eventually zero.
\end{remark}
\subsection{Proof of Theorem~\ref{thm:radon}}

Given a finitely additive measure $M$ on $\Gamma$, define the subspace
\begin{align*}
    L_M=\{f\in \ell^{\infty}(\Gamma)\,:\,M(\supp f)<\infty\} \subseteq \ell^\infty(\Gamma).
\end{align*}
Note that this is not a closed subspace. We equip it with the direct limit topology induced by the subspaces $(\ell^\infty(F))_F$, where $F$ ranges over all subsets of $\Gamma$ such that $M(F)<\infty$. 
\begin{lemma}
    \label{lem:functional}
    Let $M$ be a finitely additive measure on $\Gamma$. Then there is a continuous positive linear functional $I\colon L_M\to \R$ with $I(1_{E})=M(E)$ for all $E\subseteq \Gamma$ such that $M(E)<\infty$. If $M$ is $\mu$-stationary then $I$ is also $\mu$-stationary, i.e., $[\mu * I](f) = \sum_g \mu(g) I(g^{-1}f) = I(f)$ for all $f \in L_M$.
\end{lemma}
\begin{proof}
For each set $F\subseteq \Gamma$ with $M(F)<\infty$, let $L_{F}=\{f\in \ell^{\infty}(\Gamma)\,:\,\supp f \subseteq F\}$. Notice that the collection $\{L_{F}\}_{F}$ is partially ordered by set inclusion, and that $\bigcup_F L_F = L_M$. 

Since the restriction of $M$ to subsets of $F$ is a finite, finitely additive measure, there exists a unique positive, bounded linear functional on $\ell^\infty(F)$ that represents $M$ (see, e.g., \cite[Corollary 14.11]{guide2006infinite}). Since $L_F$ is isomorphic to $\ell^\infty(F)$ by the restriction map, it follows that there is a unique positive linear functional $I_F \colon  \to L_F$ such that $I_F(1_E) = M(E)$ for all $E \subseteq F$. Define $I \colon L_M \to R$ by $I(f) = I_{\supp f}(f)$. Then $I$ is linear because if $F_1 \subseteq F_2$ then $I_{F_2}$ extends $I_{F_1}$. It is clear that $I$ is positive. Continuity in the dual of direct limit topology on $L_M$ follows by construction.

If $M$ is $\mu$-stationary then
\begin{align*}
    \sum_g \mu(g) I(g^{-1} 1_E)
    = \sum_g \mu(g) M(g^{-1} E)
    = [\mu * M](E)
    = M(E) = I(1_E).
\end{align*}
The result follows from the linearity of $I$, and from the fact that simple functions are dense in $\ell^\infty$.
\end{proof}

\begin{proof}[Proof of Theorem~\ref{thm:radon}]
By assumption, there is a compact $K\subseteq X$ such that $\bigcup_{g\in \Gamma}g K=X$. By perhaps choosing a larger compact set $K$, we may assume that $\bigcup_{g\in \Gamma}g \interior{K}=X$, where $\interior{K}$ denotes the interior of $K$.

Fix some $x \in X$. For $Y \subset X$ define $A(Y)=\{g\in \Gamma\,:\, g x\in Y\}$. Note that the map $Y \mapsto A(Y)$ is equivariant: $A(g Y) = g A(Y)$. It also satisfies $A(\cup_i Y_i) = \cup_i A(Y_i)$ and is monotone so that if $Y_1 \subseteq Y_2$ then $A(Y_1) \subseteq A(Y_2)$. 

By Theorem~\ref{thm:subsets}, there exists a $\mu$-stationary, finitely additive measure $M$ on $\Gamma$ with $M(A(K))=1$. Given $f \in C_c(X,\R)$, let $\hat f\in \ell^{\infty}(\Gamma)$ be given by $\hat f(g)= f(g x)$. Note that for $h \in \Gamma$
\begin{align*}
    \widehat{[h f]}(g) = [h f](g x) = f(h^{-1}g x) = \hat f(h^{-1}g) = [h \hat f](g).
\end{align*}

Let $L_M \subseteq \ell^\infty(\Gamma)$ be the linear space defined in the statement of Lemma~\ref{lem:functional}. We claim that $\hat f \in L_M$. To see this, suppose that the support of $f$ is contained in a compact set $K'\subseteq X$. Then there exists a finite $F\subseteq \Gamma$ such that $K'\subseteq \bigcup_{g\in F}g \interior{K}$, by compactness and since $g \interior{K}$ is open. Then 
\begin{align*}
    \supp \hat f
    = \{g \in \Gamma\,:\, f(g x) \neq 0\}
    \subseteq M(\{g \in \Gamma\,:\, g x \in K'\})
    = A(K').
\end{align*}
Since the map $Y \mapsto A(Y)$ is monotone, it follows that 
\begin{align*}
    M(\supp \hat f)
    &\le M\left(A\left(\cup_{g\in F}g \interior{K}\right)\right)\\
    &= M\left(\cup_{g \in F} A(g \interior{K})\right)\\
    &\leq \sum_{g \in F} M(A(g \interior{K}))\\
    &= \sum_{g \in F} M(g A(\interior{K})).
\end{align*}
By $\mu$-stationarity, for any $E \subseteq \Gamma$ and $g \in \Gamma$, $M(g E) \leq \frac{1}{\max_n\mu^{(n)}(g)}M(E)$, which is finite because $\mu^{(n)}(g) > 0$ for some $n$. Hence
\begin{align*}
    M(\supp \hat f)
    \le \frac{1}{\min_{g \in F}\max_n\mu^{(n)}(g)}\cdot |F|\cdot M(A(K))
    <\infty.
\end{align*}
Hence $\hat f \in L_M$. We further claim that the map $f \mapsto \hat f$ is continuous, in the direct limit topologies on $L_M$ and $C_c(X,\R)$. This follows from the discussion above, which shows that restricted to functions $f$ supported on some compact $K' \subseteq X$, this map is an isometry into $\ell^\infty(A(K'))$.

Let $I\colon L_M\to \R$ be the linear functional established in Lemma~\ref{lem:functional}. We define a  linear functional $\hat I\colon C_c(X,\R)\to \R$ by $\hat I(f)=I(\hat{f})$. We claim that $\hat I$ is $\mu$-stationary by the $\mu$-stationarity of $I$. To see this, observe that
\begin{align*}
  [\mu * \hat I](f)
  = \sum_{g}\mu(g)\hat I(g^{-1}f)
  =\sum_{g}\mu(g)I(\widehat{[g^{-1}f]})
  =\sum_{g}\mu(g)I(g^{-1}\hat f),
\end{align*}
which is equal to $I(\hat{f})$ (by the $\mu$-stationarity of $I$), which, by definition, is equal to $\hat I(f)$. We claim that $\hat I$ is continuous in the standard, direct limit topology on $C_c(X,\R)$. This follows from the continuity of $I$ in the direct limit topology on $L_M$, and the fact that the map $f \mapsto \hat f$ is continuous.

Thus, by the Riesz-Markov-Kakutani representation theorem, there is a Radon measure $\lambda$ such that $\hat I(f)= \int_{X} f \,\dd\lambda$ for any $f\in C_c(X,\R)$. We claim $\lambda$ is $\mu$-stationary and nonzero. Since $\hat I$ is $\mu$-stationary, we get 
\begin{align*}
    \lambda(E)=\int_{X} 1_{E}d\lambda &= \hat I(1_{E})\\ 
    &=[\mu*\hat I](1_E) \\ 
    &=\sum_{g} \mu(g)\hat I(1_{g^{-1}E}) \\ 
    &= \sum_{g} \mu(g)\int_{X} 1_{g^{-1}E}d\lambda\\ 
    &=[\mu*\lambda](E).
\end{align*}
Let $\epsilon > 0$. By outer regularity of $\lambda$, there exists an open set $U\supseteq K$ such that $\lambda(U)\le \lambda(K) + \epsilon$. Because $C_c(X,\R)$ is $\ell^{\infty}$ dense in the space of functions $X\to \mathbb{R}$, there exists a function $f \in C_c(X,\R)$ such that $1_{K}\le f \le 1_{U}$. Then $\hat{f}\ge \widehat{1_{K}} = 1_{A(K)}$ whence 
\begin{align*}
    \hat I(f)=I(\hat{f})\ge I(1_{A(K)})=M(A(K))=1.
\end{align*}
Then 
\begin{align*}
    \lambda(K)+\epsilon\ge \lambda(U) = \hat I(1_{U})\ge \hat I(f)\ge 1.
\end{align*}
Choosing any $0 < \epsilon  < 1$, we find that $\lambda(K)>0$. In particular, $\lambda$ is nontrivial.
\end{proof}

\subsection{$\mu$-light subsets of $\Gamma$}

For $d \geq 2$, let $\Gamma = \bF_d = \langle S \rangle$ be the free group generated by $S = \{a_1,\ldots,a_d\}$. Denote the geometric boundary of $\bF_d$ by $\partial \Gamma$. This is the set of infinite reduced words $s_1s_2\cdots \in (S \cup S^{-1})^\N$, where each $s_i \in S \cup S^{-1}$, and $s_i \neq s_{i+1}^{-1}$. Each $w \in \partial \Gamma$ defines a geodesic in the Cayley graph defined by $S$, starting at $e \in \Gamma$ and given by $(e, s_1, s_1 s_2, s_1 s_2 s_3,\ldots)$.

Given $w=s_1s_2\cdots \in \partial \Gamma$ and $g \in \Gamma$, denote the distance between $g$ and the geodesic defined by $w$
\begin{align*}
  D(g,w) = \min_{i \geq 0}|g^{-1} \cdot s_1s_2 \cdots s_i|.
\end{align*}
Here $|\cdot|$ denotes the word length norm on $\Gamma$ defined by the generating set $S$. We can identify $w$ with the function $f_w \colon \Gamma \to \R$ given by
\begin{align}
\label{eq:fw}
  f_w(g) = (2d-1)^{|g^{-1}|- 2D(g^{-1},w)}.
\end{align}
This is the function that attains the value $(2d-1)^i$ on the inverse of the $i$\textsuperscript{th} element of the geodesic defined by $w$, and decays exponentially by a factor of $2d-1$ as one moves away from this geodesic (in the left Cayley graph). Note that $f_w$ is positive and (left) $\mu$-harmonic.

The functions $\{f_w\}_{w \in \partial \Gamma}$ form the Martin boundary of the $\mu$-random walk on $\Gamma$ (see, e.g., \cite{ancona2007positive}). This implies that they contain (and, in the case of the free group, are equal to) the extreme points of the set of positive $\mu$-harmonic functions that attain unity at the identity. Hence, for every positive $\mu$-harmonic function $f$ there exists a probability measure $\zeta$ on $\partial \Gamma$ such that
\begin{align*}
  f = f(e)\int_{\partial \Gamma}f_w\,\dd\zeta(w).
\end{align*}
It follows that if $f(A) < \infty$ then $f_w(A) < \infty$ for $\zeta$-almost every $w$. In particular, there is some $w$ such that $f_w(A) < \infty$. Thus, if $A \subseteq \Gamma$ is $\mu$-light, this is witnessed by some $f_w$ satisfying $f_w(A) < \infty$. So to show that $A$ is not $\mu$-light it suffices to show that $f_w(A) = \infty$ for all $w \in \partial \Gamma$.

Consider the set $A = \sigma(\Gamma)$, where $\sigma \colon \Gamma \to \Gamma$ is given by \eqref{eq:sigma} (or any other function satisfying the conditions in the paragraph before \eqref{eq:sigma}). Since $|\sigma(g)| = 2|g|$, we have that $|A \cap B_{2r}| = |B_r| = (2d)(2d-1)^{r-1}$. Hence $\underline{\mathfrak{g}}=\overline{\mathfrak{g}}=\sqrt{2d-1}$. Given any $w=s_1s_2\cdots \in \partial \Gamma$, let $g_i = (s_1s_2\cdots s_i)^{-1}$, note that  $|\sigma(g_i)| = 2i$, and that $D(g_i^{-1},w) \leq i$. Hence $f_w(\sigma(g_i)) \geq 1$ and since $\sigma(g_i) \in A$, $f_w(A) \ge \sum_{i\ge 1} f_w(\sigma(g_i)) = \infty$. This completes the proof of Claim~\ref{clm:sigma}.

\begin{proof}[Proof of Proposition~\ref{prop:small}]
  Let $\zeta$ be the hitting measure of the $\mu$-random walk on $\partial \Gamma$, also known as the harmonic measure on $\partial \Gamma$. This is the probability measure in which we draw $w = s_1s_2\cdots$ by choosing $s_1$ uniformly at random from $S \cup S^{-1}$, and conditioned on $s_1,s_2,\ldots,s_i$, choose $s_{i+1}$ uniformly from $(S \cup S^{-1})\setminus s_i^{-1}$. Formally, given a finite reduced word $t_1t_2\cdots t_n$ with each $t_i \in S \cup S^{-1}$, 
  \begin{align*}
    \zeta(\{w = s_1s_2\cdots \,:\, s_1=t_1, \ldots, s_n=t_n\}) = 1/\left(2d\cdot (2d-1)^{n-1}\right).
  \end{align*}

  Suppose $\overline{\mathfrak{g}}(A) < \sqrt{2d-1}$. We claim that for $\zeta$-almost every $w$ it holds that $f_w(A) < \infty$, which in particular proves that there exists a measure with the required properties, so that $A$ is $\mu$-light. To this end, we show that
  \begin{align}
    \label{eq:zeta}
    \int_{\partial \Gamma}\sum_{g \in A}\sqrt{f_w(g)}\,\dd\zeta(w) < \infty.
  \end{align}
  This implies that $\sum_{g \in A}\sqrt{f_w(g)} < \infty$, $\zeta$-almost surely, which implies that $f_w(A) = \sum_{g \in A}f_w(g) < \infty$, $\zeta$-almost surely.

  Denote $S_r  =\{ g\in \Gamma\,:\, |g|=r\}$. Fix some $g \in S_r$. Then 
  \begin{align}
    \label{eq:int-sum-zeta}
    \int_{\partial \Gamma} \sqrt{f_w(g)}\,\dd \zeta(w) = \frac{1}{|S_r|}\sum_{h \in S_r}\sqrt{f_w(h)},
  \end{align}
  for any $w \in \partial \Gamma$, by the spherical symmetry of the measure $\zeta$.

  Now, on $S_r$, $f_w$ attains its maximum $(2d-1)^r$ at the single point $h_0$ which coincides with the $r$-prefix of $w$. At any other point $h$,
  \begin{align*}
    f_w(h) = f_w(h_0) (2d-1)^{-|h h_0^{-1}|} = (2d-1)^{r-|h h_0^{-1}|},
  \end{align*}
  Now, the number of elements $h \in S_r$ such that $|h h_0^{-1}|=t$ is $0$ if $t$ is odd or larger than $2r$, and $(2d-1)^{t/2}$ otherwise. Hence
  \begin{align*}
    \sum_{h \in S_r}\sqrt{f_w(h)}
    &= \sum_{t \in \{0,2,\ldots,2r\}}|S_r \cap \{h \,:\, |h h_0^{-1}|=t\}|(2d-1)^{r/2-t/2}\\
    &= \sum_{t \in \{0,2,\ldots,2r\}}(2d-1)^{t/2}(2d-1)^{r/2-t/2}\\
    &= (2d-1)^{r/2}r,
  \end{align*}
  so that 
  \begin{align*}
    \frac{1}{|S_r|}\sum_{h \in S_r}\sqrt{f_w(h)} = 2d(2d-1)^{-(r-1)}(2d-1)^{r/2}r = 2d(2d-1)^{-r/2+1} r.
  \end{align*}
  It follows by \eqref{eq:int-sum-zeta} that 
  \begin{align*}
    \int_{\partial \Gamma}\sum_{g \in A}\sqrt{f_w(g)}\,\dd\zeta(w)
    &= \sum_r \sum_{g \in A \cap S_r}\int_{\partial \Gamma}\sqrt{f_w(g)}\,\dd\zeta(w)\\
    &= \sum_r \sum_{g \in A \cap S_r}2d(2d-1)^{-r/2+1} r\\
    &= \sum_r |A \cap S_r|2d(2d-1)^{-r/2+1} r,
  \end{align*}
  which is finite if $\overline{\mathfrak{g}} < \sqrt{2d-1}$. Thus \eqref{eq:zeta} holds and the proof is complete.
\end{proof}

\subsection{Necessity of co-compactness in Theorem~\ref{thm:radon}}

Recall from \S\ref{sec:small} that $A \subseteq \Gamma$ is $\mu$-light if $f(A) = \sum_{g \in A}f(g) < \infty$ for some positive, left $\mu$-harmonic $f$. \begin{lemma} \label{lem:seq}
Let $X$ be a locally compact Hausdorff $\Gamma$-space. Suppose that there exists a sequence of compact subsets $K_n \subseteq X$, and a sequence of subsets $A_n \subseteq \Gamma$  with the following properties:
\begin{enumerate}[(i)]
    \item Each $A_n$ is not $\mu$-light.
    \item $(g^{-1} K_n)_{g\in A_n}$ are pairwise disjoint subsets of $K_{n+1}$.
    \item $X=\cup_n\cup_{g \in \Gamma}gK_n$.
\end{enumerate}
Then there is no nonzero $\mu$-stationary Radon measure on $X$. 
\end{lemma}
\begin{proof}
    Suppose $\lambda$ is a $\mu$-stationary Radon measure on $X$. Since $\lambda$ is Radon and $K_n$ is compact, $\lambda(K_n)<\infty$. Note
    that $f(g) = \lambda(g K_n)$ is (left) $\mu$-harmonic. Suppose $\lambda(K_n) > 0$, so that $f$ is positive. Then (\textit{i}) implies that
    \begin{align*}
        \sum_{g \in A_n}\lambda(g K_n) = f(A_n) = \infty.
    \end{align*}
    But then, by (\textit{ii}), $\lambda(K_{n+1}) \geq \sum_{g \in A_n}\lambda(g K_n) = \infty$, which is impossible because $\lambda$ is Radon. Hence we have that $f=0$, and so also $\lambda(g K_n)=0$ for all $g \in \Gamma$. Finally, by (\textit{iii}),  $\lambda(X)\le \sum_{n,g\in A_n}\lambda(gK_n)=0$.
\end{proof}
\begin{lemma}\label{lem:subsets}
There exist subsets $A_1,A_2,\ldots$ of $\mathbb{F}_d$ such that, for every $n$, $e\in A_n$, each $A_n$ is not $\mu$-light, and the map $\psi_n\colon A_1\times \cdots\times A_n\to \Gamma$ given by $\psi_n(g_1,\ldots, g_n)=g_1\cdots g_n$ is injective.
\end{lemma}
\begin{proof}
Recall that we denote by $S_r$ the set of elements $\{g \in \Gamma \,:\, |g| = r\}$, where $|\cdot|$ is the standard word length norm. Let $a$ be one of the $d$ generators of $\Gamma$, so that $\mu(a) = 1/(2d)$, and let $A_a^a \subset \Gamma$ be the set of reduced words that begin and end with $a$. Note that for $g, h \in A_a^a$ it holds that 
\begin{align}
    \label{eq:additive}
    |g h| = |g| + |h|.
\end{align}

Let $r \colon \N^2 \to \N$ be an injection, and let 
\begin{align*}
    A_n = \cup_{m=1}^\infty (A_a^a \cap S_{2^{r(n,m)}})\cup \{e\}.
\end{align*}
By \eqref{eq:additive}, if $(g_1,\ldots,g_n) \in A_1 \times \cdots \times A_n$ then $|g_1 \cdots g_n| = |g_1| + \cdots+ |g_n|$, and so, by the definition of $(A_n)_n$, $g_1\cdots g_n$ determines each $|g_i|$; indeed, $|g_i|$ can be recovered from the binary representation of $|g_1 \cdots g_n|$. And since by the definition of $A_a^a$ there cannot be any cancellations in the product  $g_1\cdots g_n$, this product determines $(g_1,\cdots,g_n)$. Hence $\psi_n$ is injective.

It remains to be shown that $A_n$ is not $\mu$-light. To this end we proceed as in the proof of Claim~\ref{clm:sigma}, and complete the proof by showing that $f_w(A_n) = \infty$ for every $w \in \partial \Gamma$. Indeed, $f_w(g) \geq (2d-1)^{-|g|}$, by \eqref{eq:fw}. Now, $|A_a^a \cap S_r| \geq (2d-1)^{r-3}$, since $A_a^a$ only fixes the first and last letter.  Hence $f_w(g)(A_a^a \cap S_r) \geq (2d-1)^{-3}$, and thus $f_w(A_n) = \infty$.

%Because the $\mu$-random walk on $\mathbb{F}_d$ is transient, there exists (almost surely) an integer-valued random variable $T_r$ which denotes the first time $S_r$ is crossed. Then for each $n$, we have 
%\begin{align*}
%    \sum_{g\in A_n}\max_i\mu^{(i)}(g)&= 
%\sum_{m=1}^{\infty}\sum_{g\in A_a^a\cap S_{2^{r(n,m)}}}\max_i\mu^{(i)}(g)\\ &\ge \sum_{m=1}^{\infty}\sum_{g\in A_a^a\cap S_{2^{r(n,m)}}}\mu^{(T_{2^{r(n,m)}})}(g)\\ &=\sum_{m=1}^{\infty}\frac{|A_a^a\cap S_{2^{r(n,m)}}|}{| S_{2^{r(n,m)}}|} \\ 
%&= \sum_{m=1}^{\infty}\frac{1}{16}=\infty
%\end{align*}

\end{proof}

Given these lemmas, the proof of Proposition~\ref{prop:free-no-stationary} follows the construction in \cite[Proposition 4.3]{matui2015universal}.
\begin{proof}[Proof of Proposition~\ref{prop:free-no-stationary}]

%We denote the Stone–{\v C}ech compactification of $\Gamma$ by $\beta \Gamma$ and the closure of a subset $A\subseteq \Gamma$ in $\beta \Gamma$ by $\mathcal{K}(A)$. Note that $\mathcal{K}(A)$ is compact and open. Note also that if $A,B \subseteq \Gamma$ are disjoint then $\mathcal{K}(A)$ and $\mathcal{K}(B)$ are disjoint.

Let $A_1,A_2,\ldots$ be a sequence furnished by 
Lemma ~\ref{lem:subsets}. We define  $B_n=\psi_n(A_1,\dots, A_n)=A_1  \cdots A_n$. Note that this sequence is increasing, since $e \in A_n$. Let $\mathcal{A}$ be the boolean algebra generated by $(A_n)_n$, $(B_n)_n$, and by their translates by elements of $\Gamma$. Let $\mathcal{S}$ be the Stone space associated with $\mathcal{A}$. Since $\mathcal{A}$ is countable, $\mathcal{S}$ is second countable, and moreover isomorphic to the Cantor space. And since $\mathcal{A}$ admits a $\Gamma$-action by translations, so does $\mathcal{S}$.

Given $A \in \mathcal{A}$, denote by $\mathcal{K}(A) \subseteq \mathcal{S}$ the natural identification of $A$ with the subset of $\mathcal{S}$ consisting of the ultrafilters that contain $A$. Note that $\mathcal{K}(A)$ is compact and open. Note also that if $A,B \in \mathcal{A}$ are disjoint then $\mathcal{K}(A)$ and $\mathcal{K}(B)$ are disjoint.

Let
\begin{align*}
    X = \bigcup_n \bigcup_{g \in \Gamma} g\mathcal{K}(B_n).
\end{align*}
Note that $X$, as a countable union of compact open sets, is an open subset of $\mathcal{S}$, and hence locally compact, second countable and $\sigma$-compact. It is also invariant to the $G$-action, by construction.

We now show that $K_n=\mathcal{K}(B_n)\subseteq X$ and $A_{n}\subseteq \Gamma$ satisfy the conditions of Lemma~\ref{lem:seq}, which will complete the proof. Note that (\textit{i}) is satisfied by our choice of $A_n$. Further, if $g,h$  are distinct elements of $A_{n+1}$, then $g B_n$ and $h B_n$ are disjoint,  by the injectivity of $\psi_n$. It follows that $gK_{n} = g\mathcal{K}(B_n) = \mathcal{K}(g B_n)$ and $h K_{n}$ are disjoint, which shows (\textit{ii}). Finally, (\textit{iii}) follows immediately from the definition of $X$.

\end{proof}

\subsection*{Acknowledgements} This work was supported in part by NSF grants DMS-2246494 (TH), DMS-1944153 (OT), and DMS-2348143 (TZ). We thank Nicolas Monod and an anonymous referee for helpful comments, and in particular for suggesting the current, stronger version of Propsition~\ref{prop:free-no-stationary}.

\bibliography{refs.bib}

\begin{bibdiv}
\begin{biblist}

\bib{guide2006infinite}{book}{
      author={Aliprantis, Charalambos~D.},
      author={Border, Kim~C.},
       title={Infinite dimensional analysis, a hitchhiker's guide},
     edition={3},
   publisher={Springer},
        date={2006},
}

\bib{ancona2007positive}{inproceedings}{
      author={Ancona, Alano},
       title={Positive harmonic functions and hyperbolicity},
organization={Springer},
        date={2007},
   booktitle={Potential theory surveys and problems: Proceedings of a
  conference held in prague, july 19--24, 1987},
       pages={1\ndash 23},
}

\bib{Bougerol1995}{article}{
      author={Bougerol, Philippe},
      author={Elie, Laure},
       title={Existence of positive harmonic functions on groups and on
  covering manifolds},
        date={1995},
     journal={Annales de l'I.H.P. Probabilit{\'e}s et statistiques},
      volume={31},
      number={1},
       pages={59\ndash 80},
}

\bib{deroin2013symmetric}{article}{
      author={Deroin, Bertrand},
      author={Kleptsyn, Victor},
      author={Navas, Andr{\'e}s},
      author={Parwani, Kamlesh},
       title={Symmetric random walks on $\mathrm{Homeo}^+(\mathbb{R})$},
        date={2013},
     journal={The Annals of Probability},
      volume={41},
      number={3},
       pages={2066\ndash 2089},
}

\bib{hutchcroft2019statistical}{article}{
      author={Hutchcroft, Tom},
       title={Statistical physics on a product of trees},
        date={2019},
     journal={Ann. Inst. H. Poincar{\'e} Probab. Statist.},
      volume={55},
       pages={1001\ndash 1010},
}

\bib{kellerhals2013non}{article}{
      author={Kellerhals, Julian},
      author={Monod, Nicolas},
      author={R{\o }rdam, Mikael},
       title={Non-supramenable groups acting on locally compact spaces},
        date={2013},
     journal={Doc. Math.},
      volume={18},
       pages={1597\ndash 1626},
}

\bib{margulis1966positive}{inproceedings}{
      author={Margulis, Grigorii~Aleksandrovich},
       title={Positive harmonic functions on nilpotent groups},
organization={Russian Academy of Sciences},
        date={1966},
   booktitle={Doklady akademii nauk},
      volume={166},
       pages={1054\ndash 1057},
}

\bib{matui2015universal}{article}{
      author={Matui, Hiroki},
      author={R{\o}rdam, Mikael},
       title={Universal properties of group actions on locally compact spaces},
        date={2015},
     journal={Journal of Functional Analysis},
      volume={268},
      number={12},
       pages={3601\ndash 3648},
}

\bib{rosenblatt1974invariant}{article}{
      author={Rosenblatt, Joseph~Max},
       title={Invariant measures and growth conditions},
        date={1974},
     journal={Transactions of the American Mathematical Society},
      volume={193},
       pages={33\ndash 53},
}

\bib{tarski1949car}{article}{
      author={Tarski, Alfred},
       title={Cardinal algebras},
        date={1949},
     journal={Oxford University Press},
       pages={233},
}

\bib{vneumann1929zur}{article}{
      author={von Neumann, John},
       title={Zur allgemeinen theorie des ma{\ss}es},
        date={1929},
     journal={Fund. Math.},
      volume={13},
       pages={73\ndash 116},
}

\end{biblist}
\end{bibdiv}

\end{document}